\documentclass[12pt,letterpaper]{article}
\usepackage{amssymb,amsmath,amsthm}
\usepackage{subcaption}
\usepackage{graphicx}
\usepackage[numbers,square]{natbib}
\usepackage{sectsty}
\sectionfont{\normalsize}
\subsectionfont{\normalsize}
\subsubsectionfont{\normalsize}
\paragraphfont{\normalsize}
\newtheorem{thm}{Theorem}
\newtheorem*{thm*}{Theorem}
\newtheorem{lem}[thm]{Lemma}

\newtheorem{cor}[thm]{Corollary}
\newtheorem{clm}[thm]{Claim}
\newcommand{\llll}[1] {\left #1}
\newcommand{\rrrr}[1] {\right #1}

\newcommand{\dddd}[2]{\dfrac{#1}{#2}}

\newcommand{\aaaa}{\alpha}
\newcommand{\tttt}{\tau}
\newcommand{\ssss}{\sigma}

\newcommand{\dddddd}{\delta}
\newcommand{\bbbb}{\beta}
\newcommand{\GGGG}{\Gamma}
\newcommand{\zzzz}{\zeta}
\newcommand{\eeee}{\epsilon}
\newcommand{\rrrrrrrr}{\rightarrow}
\begin{document}
\nocite{*}
\title{{\bf \normalsize A SECOND ORDER APPROXIMATION FOR THE CAPUTO FRACTIONAL DERIVATIVE }}
\author{Yuri Dimitrov\\
Department of Applied Mathematics and Statistics \\
University of Rousse, Rousse  7017, Bulgaria\\
\texttt{ymdimitrov@uni-ruse.bg}}
\maketitle
\begin{abstract} 
When  $0<\aaaa<1$, the approximation for the Caputo derivative  
$$y^{(\aaaa)}(x)=\dddd{1}{\GGGG(2-\aaaa)h^\aaaa}\sum_{k=0}^n \ssss_k^{(\aaaa)} y(x-kh)+O\llll(h^{2-\aaaa}\rrrr),
$$
where $\sigma_0^{(\aaaa)}=1, \ssss_n^{(\aaaa)}=(n-1)^{1-a}-n^{1-a}$ and 
$$\sigma_k^{(\aaaa)}=(k-1)^{1-\aaaa}-2k^{1-a}+(k+1)^{1-\aaaa},\quad (k=1\cdots,n-1),$$
has accuracy $O\llll(h^{2-\aaaa}\rrrr)$.
We use the  expansion of  $\sum_{k=0}^n k^\aaaa$ to determine an approximation for the fractional integral of order $2-\aaaa$ and the second order approximation for the Caputo derivative 
$$y^{(\aaaa)}(x)=\dddd{1}{\GGGG(2-\aaaa)h^\aaaa}\sum_{k=0}^n \delta_k^{(\aaaa)} y(x-kh)+O\llll(h^{2}\rrrr),$$
where $\dddddd_k^{(\aaaa)}=\ssss_k^{(\aaaa)}$ for $2\leq k\leq n$,
$$\dddddd_0^{(\aaaa)}=\ssss_0^{(\aaaa)}-\zzzz(\aaaa-1), \dddddd_1^{(\aaaa)}=\ssss_1^{(\aaaa)}+2\zzzz(\aaaa-1),\dddddd_2^{(\aaaa)}=\ssss_2^{(\aaaa)}-\zzzz(\aaaa-1),$$
and $\zzzz(s)$ is the Riemann zeta function. The numerical solutions of the fractional  relaxation and  subdiffusion equations are computed.\\
{\bf 2010 Math Subject Classification:} 26A33, 34E05, 33F05, 26A33 
\end{abstract}
\section{Introduction}\label{Intro} 
Fractional differential equations are used for modeling complex   diffusion processes in science and engineering [1--5]. 
The Caputo fractional derivatives are important as a tool for describing nature as well as for their relation to integer order derivatives and special functions. The Caputo derivative  of order $\aaaa$, when $0<\aaaa<1$, is defined as the convolution of the power function $x^{-a}$ and the first derivative of the function on the interval $[0,x]$
$$y^{(\aaaa)}(x)=D^\aaaa y(x)=\dddd{d^{\aaaa}y(x)}{d x^\aaaa}=\dddd{1}{\Gamma (1-\aaaa)}\int_0^x \dfrac{y'(\xi)}{(x-\xi)^{\aaaa}}d\xi.$$
When the function $y(x)$ is defined on the interval $(-\infty,x]$, the lower limit of the integral in the definition of Caputo derivative is $-\infty$. The Caputo derivative of the constant function $1$ is zero, and
$$D^\aaaa D^{1-\aaaa}y(x)=y^\prime(x).$$
 While the integer order derivatives describe the local behavior of a function, the fractional derivative $y^{(\aaaa)}(x)$ depends on the values of the function on the interval $[0,x]$. One approach for discretizing the Caputo derivative is to divide the interval  to subintervals of small length and approximate the values of the function on each subinterval with a Lagrange polynomial.  Let   $x_n=n h$  and $y_n=y(x_n)=y(n h)$, where $h>0$ is a small number. The Lagrange polynomial  for the function $y'(x)$  at the midpoint $x_{k-0.5}$ of the interval $[x_{k-1},x_k]$  is the value of $y^\prime(x_{k-0.5})$.

 Approximation \eqref{A1} for the Caputo fractional derivative is a commonly used approximation for numerical solutions of ordinary and partial fractional differential equations [6-8]. 
\begin{equation*}
\begin{aligned}
 \Gamma (1-\alpha) y^{(\alpha)}(x_n)&=\int_0^{x_n} \dfrac{y'(\xi)}{(x_n-\xi)^\alpha}d\xi\approx\sum_{k=1}^n\int_{x_{k-1}}^{x_k} \dfrac{y'(x_{k-0.5})}{(x_n-\xi)^\alpha}d\xi \\
&  \approx \sum_{k=1}^n \dfrac{y(x_k)-y(x_{k-1})}{h}\int_{(k-1)h}^{kh} \dfrac{1}{(nh-\xi)^\alpha}d\xi\\
&  = \sum_{k=1}^n \dfrac{y_k-y_{k-1}}{h}  \dfrac{((n-k+1)h)^{1-\alpha}-((n-k)h)^{1-\alpha}}{1-\alpha}.\\
\end{aligned}
\end{equation*}
Let $\rho_k^{(\alpha)}=(n-k+1)^{1-\alpha}-(n-k)^{1-\alpha}$.
\begin{align*}
\Gamma (2-\alpha)h^\alpha y^{(\alpha)}_n   &\approx \sum_{k=1}^n \rho_k^{(\alpha)}(y_k-y_{k-1})=\sum_{k=1}^n y_k\rho_k^{(\alpha)}-\sum_{k=1}^n y_{k-1}\rho_{k}^{(\alpha)}\\
&   =\rho_n^{(\alpha)} y_n+\sum_{k=1}^{n-1} y_k\left(\rho_k^{(\alpha)}-\rho_{k-1}^{(\alpha)}\right)-\rho_1^{(\alpha)}y_0.
\end{align*}
Then
$$y^{(\alpha)}_n\approx \dddd{1}{\Gamma (2-\alpha)h^\alpha}\left(\rho_n^{(\alpha)} y_n+\sum_{k=1}^{n-1} y_{n-k}\left(\rho_{n-k}^{(\alpha)}-\rho_{n-k+1}^{(\alpha)} \right)-\rho_1^{(\alpha)}y_0\right).$$
Let $\ssss_0^{(\alpha)}=\rho_n^{(\alpha)}=1$,  $\ssss_n^{(\alpha)}=-\rho_1^{(\alpha)}=(n-1)^{1-a}-n^{1-a}$ and
$$\ssss_k^{(\alpha)}=\rho_{n-k}^{(\alpha)}-\rho_{n-k+1}^{(\alpha)}=(k+1)^{1-\alpha}-2k^{1-\alpha}+(k-1)^{1-\alpha},$$
for $k=1,2,\cdots,n-1$. Denote
$$\mathcal{A}_{h}y_n= \sum_{k=0}^{n} \ssss_k^{(\alpha)} y_{n-k}. $$
 We obtain the approximation for the Caputo derivative
\begin{equation}
y^{(\alpha)}_n  \approx\dfrac{1}{\GGGG(2-\aaaa)h^\alpha}\mathcal{A}_{h}y_n.
\label{A1}
\end{equation}
Approximation \eqref{A1} has accuracy $O(h^{2-\alpha})$ when $y\in C^2[0,x_n]$ (\cite{LinXu2007}). 
\begin{table}[ht]
	\caption{Error and order of approximation \eqref{A1} for     $y(x)=\cos x$ on the interval $[0,1]$, when $\aaaa=0.6$.}
	\centering
  \begin{tabular}{ l  c c c }
    \hline \hline
    $h$  & $Error$                 & $Ratio$ & $Order$  
		\\ \hline \hline
$0.05$     & $0.0023484$          & $2.69618$   & $1.43092$   \\ 
$0.025$    & $0.000878437$        & $2.67338$   & $1.41867$   \\ 
$0.0125$   & $0.000330265$        & $2.65979 $  & $1.41131$   \\ 
$0.00625$  & $0.000124548$        & $2.65171$   & $1.40692$   \\ 
$0.003125$ & $0.0000470549$       & $2.64687$   & $ 1.40429$   \\ \hline
  \end{tabular}
	\end{table}
	
 The numbers $\ssss_k^{(\alpha)}$ have the following  properties:
\begin{equation*}
\ssss_0^{(\alpha)}>0,\quad \ssss_1^{(\alpha)}<\ssss_2^{(\alpha)}<\cdots<\ssss_k^{(\alpha)}<\cdots<0,\quad \sum_{k= 0}^\infty \ssss_k^{(\alpha)} = 0.
\end{equation*}

Approximation \eqref{A1} and its modifications have been successfully used for numerical solutions of fractional differential equations, as well as in proofs of the convergence of numerical methods. One disadvantage of \eqref{A1} is that when the order of the Caputo fractional derivative $\aaaa \approx 1$, its accuracy decreases to $ O(h)$. The numerical solutions of multidimensional partial fractional differential equations require a large number of computations,  when the approximation has accuracy $O(h)$. 

In section 4, we determine the second order approximation \eqref{A2} for the Caputo derivative by modifying the  first three coefficients of \eqref{A1} with values of the Riemann zeta function. Approximation \eqref{A2} has accuracy $O\llll(h^2\rrrr)$ for all values of $\aaaa$ between $0$ and $1$. 

The ordinary fractional differential equation
\begin{equation}\label{freqn}
y^{(\aaaa)}+By=F(t),
\end{equation}
is called relaxation equation when $0<\aaaa<1$, and oscillation equation when $1<\aaaa<2$. 
 In section 5 we compare the numerical solutions for the relaxation and the time-fractional subdiffusion equations for discretizations \eqref{A1} and \eqref{A2}.  We observe a noticeable improvement of the accuracy of the numerical solutions using approximation \eqref{A2} for  Caputo derivative,  especially when $\aaaa \approx 1$.

When $y(x)$ is a sufficiently differentiable function, the integral in the definition of the Caputo derivative has a singularity at the point $x$. Sidi \cite{Sidi2004} discusses approximations for integrals with singularities. 

The sum of the powers of the first $n-1$ integers has  expansion \cite{AbramowitzStegun1964}
\begin{equation}\label{Expansion1}
\sum_{k=1}^{n-1} k^\aaaa = \zzzz(-\aaaa)+\dddd{n^{\aaaa+1}}{\aaaa+1}\sum_{m=0}^\infty \binom{\aaaa+1}{m}\dddd{B_{m}}{n^m},
\end{equation}
where $\aaaa \neq -1$ and $B_m$ are the Bernoulli numbers. In section 3, we use expansion \eqref{Expansion1} to determine a second order approximation \eqref{A3} for the left Riemann sums and the fractional integral of order $2-\aaaa$.
In section 4 we determine the second order approximation  for the Caputo derivative  \eqref{A2} from  \eqref{A3},  using discrete integration by parts and second order backward difference approximation for the second derivative. 
\section{Preliminaries}
In this section we introduce the basic definitions and facts used in the paper. The fractional integral of order $\aaaa$ is defined  as the convolution of the function $y(x)$ and the power function $x^{\aaaa-1}$ on the interval $[0,x]$
$$J^\aaaa y(x)=\dddd{1}{\Gamma (\aaaa)}\int_0^x \dfrac{y(\xi)}{(x-\xi)^{1-\aaaa}}d\xi,$$
 where $\aaaa>0$. The fractional integral of order $\aaaa$ is often denoted as $y^{(-\aaaa)}(x)$. The value of the fractional integral of order $\aaaa$ of the constant function $1$ is $x^\aaaa/\GGGG(\aaaa+1)$. The Caputo derivative is defined as the composition of $y^\prime(x)$ with a fractional integral of
of order $1-a$. In Claim 1, we represent the Caputo derivative with the composition of the second derivative $y^{\prime \prime}(x)$ and a fractional integral of order $2-\aaaa$
$$y^{(\aaaa)}(x)=J^{1-\aaaa}y^\prime (x)=J^{2-\aaaa}y^{\prime \prime} (x)+\dddd{y^{\prime}(0)x^{1-\aaaa}}{\GGGG(2-\aaaa)}.$$
The composition of fractional integrals satisfies
$$J^\aaaa J^\bbbb y(x)=J^\bbbb J^\aaaa y(x)=J^{\aaaa+\bbbb}y(x).$$
The composition of the Caputo derivative and the fractional integral of order $\aaaa$, when $0<\aaaa<1$, has properties
$$D^{\aaaa}J^{\aaaa}y(x)=y(x),\quad J^{\aaaa}D^{\aaaa}y(x)=y(x)-y(0).$$
In Theorem 3 we use the expansion of the  sum of the powers of the first $n-1$ integers \eqref{Expansion1}, to determine the  second order approximation for the fractional integral of order $2-\aaaa$
$$\dddd{h^{2-\aaaa}}{\GGGG(2-\aaaa)}\sum_{k=1}^n k^{1-\aaaa}y(x-k h) \approx J^{2-\aaaa}y(x)+\dddd{y(0)}{2\GGGG(2-\aaaa)}x^{1-\aaaa}h+
\dddd{\zzzz(\aaaa-1)}{\GGGG(2-\aaaa)}y(x)h^{2-\aaaa},$$
where $h=x/n$, and $\zzzz(s)$ is the Riemann zeta function, defined as the analytic continuation of the function
$$\zzzz(s)=1+\dddd{1}{2^s}+\dddd{1}{3^s}+\cdots+\dddd{1}{n^s}+\cdots=\sum_{n=1}^\infty {n^{-s}}  \quad ( Re(s)>1).$$
In the special case of \eqref{Expansion1}, when $\aaaa=-1$, the  sums of the harmonic series have expansion \cite{AbramowitzStegun1964} 
$$\sum_{k=1}^{n-1} \dddd{1}{k} \approx \ln n+\gamma-\dddd{1}{2n}-\sum_{s=1}^\infty \dddd{B_{2m}}{2m}\dddd{1}{n^{2m}},$$
where $\gamma\approx 0.5772$ is the Euler-Mascheroni constant and $B_{2m}$ are the Bernoulli numbers.

In section 4  we use the approximation for the fractional integral \eqref{A3}, to determine the second order approximations for the Caputo fractional derivative
\begin{equation*}
y^{(\aaaa)}(x)=\dddd{1}{\GGGG(2-\aaaa)h^\aaaa}\sum_{k=0}^n \ssss_k^{(\aaaa)} y(x-kh)-\dddd{\zzzz(\aaaa-1)}{\GGGG(2-\aaaa)}y^{\prime\prime}(x)h^{2-\aaaa}+O\llll(h^{2}\rrrr),
\end{equation*} 
\begin{equation}\label{A2}
y^{(\aaaa)}(x)=\dddd{1}{\GGGG(2-\aaaa)h^\aaaa}\sum_{k=0}^n \delta_k^{(\aaaa)} y(x-kh)+O\llll(h^{2}\rrrr).
\end{equation}
The numbers $\dddddd_k^{(\aaaa)}$ are computed from the coefficients $\ssss_k^{(\aaaa)}$ of  \eqref{A1} by
$$\dddddd_0^{(\aaaa)}=\ssss_0^{(\aaaa)}-\zzzz(\aaaa-1),\; \dddddd_1^{(\aaaa)}=\ssss_1^{(\aaaa)}+2\zzzz(\aaaa-1),\; \dddddd_2^{(\aaaa)}=\ssss_2^{(\aaaa)}-\zzzz(\aaaa-1),$$
$$\dddddd_k^{(\aaaa)}=\ssss_k^{(\aaaa)}\quad ( k=2,3,\cdots, n).$$
The values of the Riemann zeta function  satisfy  \cite{Havil2003}
$$\zzzz(s)=\dddd{1}{1-2^{1-s}}\sum_{n=0}^\infty \dddd{1}{2^{n+1}}\sum_{k=0}^n (-1)^k\binom{n}{k}(k+1)^{-s},
$$
for all $s\in \mathbb{C}$, and  the functional equation
$$\zzzz(s)=2^s\pi^{s-1}\sin \llll(\dddd{\pi s}{2} \rrrr)\GGGG (1-s)\zzzz (1-s).$$
From the functional equation for the Riemann zeta function we obtain a representation of  $\zzzz(\aaaa-1)/\GGGG(2-\aaaa)$ 
 $$\dddd{\zzzz(\aaaa-1)}{\GGGG(2-\aaaa)}=-2^{\aaaa-1}\pi^{\aaaa-2}
\cos \llll(\dddd{\pi \aaaa}{2} \rrrr)\zzzz (2-\aaaa).
$$
\section{Approximation for the Fractional Integral of Order $2-\aaaa$}
In this section we determine a second order approximation \eqref{A3} for the fractional integral of order $2-\aaaa$, when $0<\aaaa<1$
$$J^{2-\aaaa} y(x)=\dddd{1}{\Gamma (2-\aaaa)}\int_0^x (x-\xi)^{1-\aaaa} y(\xi)d\xi.$$
Approximation \eqref{A3} uses the left Riemann sums of a uniform partition of the interval $[0,x]$, and the values of $y(0)$ and $y(x)$. The Caputo derivative $y^{(\aaaa)}(x)=J^{1-\aaaa}y^\prime(x)$ is defined as the composition of the fractional integral of order $1-\aaaa$ and the first derivative $y^\prime(x)$. In Claim 1 we use integration by parts to express the Caputo derivative as a composition of the fractional integral of order $2-\aaaa$ and the second derivative $y^{\prime\prime}(x)$. 
\begin{clm} Let $y\in C^2[0,x]$, and  $0<\aaaa<1$. 
$$\GGGG(2-\aaaa)y^{(\aaaa)}(x)=\GGGG(2-\aaaa)J^{2-\aaaa}y^{\prime\prime}(x)+y^\prime (0)x^{1-\aaaa}.$$
\end{clm}
\begin{proof} From the properties of the composition of fractional integrals and Caputo derivatives
$$J^{2-\aaaa}y''(x)=J^{1-\aaaa}J^1y''(x)=J^{1-\aaaa}(y'(x)-y'(0)).$$
Then
$$y^{(\aaaa)}(x)=J^{1-\aaaa}y'(x)=J^{2-\aaaa}y''(x)+J^{1-\aaaa}y'(0),$$
$$y^{(\aaaa)}(x)=J^{2-\aaaa}y''(x)+\dddd{y'(0)x^{1-\aaaa}}{\GGGG(2-\aaaa)}.$$
\end{proof}
Let $x=nh$, where $n$ is a positive integer. Consider the partition $\mathcal{P}_h$ of the interval $[0,x]$ to $n$ subintervals of length $h$.  Denote by $\mathcal{L}^{(\aaaa)}_{y,h}$ and $\mathcal{T}^{(\aaaa)}_{y,h}$ the left Riemann sum and the Trapezoidal sum of the function $(x-\xi)^{1-\aaaa}y(\xi)$ for  partition  $\mathcal{P}_h$
$$\mathcal{L}^{(\aaaa)}_{y,h}=h\sum_{m=0}^{n} (x-m h)^{1-\aaaa}y(m h)=h\sum_{m=0}^{n-1} (n h-m h)^{1-\aaaa}y(m h),$$
$$\mathcal{T}^{(\aaaa)}_{y,h}=\dddd{h}{2}\llll((n h)^{1-\aaaa}f(0)+2\sum_{m=1}^{n-1} (n h-m h)^{1-\aaaa}y(m h)\rrrr).$$
Substitute $k=n-m$
$$\mathcal{L}^{(\aaaa)}_{y,h}=h^{2-\aaaa}\sum_{k=1}^{n} k^{1-\aaaa}y(x-kh),$$
$$\mathcal{T}^{(\aaaa)}_{y,h}=\dddd{y(0)}{2}x^{1-\aaaa}h+h^{2-\aaaa}\sum_{k=1}^{n-1} k^{1-\aaaa}y(x-kh).$$
The numbers $\mathcal{L}^{(\aaaa)}_{y,h}$ and $\mathcal{T}^{(\aaaa)}_{y,h}$ are approximations for $\GGGG(2-\aaaa)J^{(2-\aaaa)}y(x)$ and 
\begin{equation}\label{L1}
\mathcal{L}^{(\aaaa)}_{y,h}-\mathcal{T}^{(\aaaa)}_{y,h}=\dddd{y(0)}{2}x^{1-\aaaa}h.
\end{equation}
Now we use  \eqref{Expansion1}, to determine a second order approximation for the left Riemann sums of the constant function $y(x)=1$.
\begin{lem} Let $x=n h$, where $n$ is a positive integer.
$$\mathcal{L}^{(\aaaa)}_{1,h}=\dddd{x^{2-\aaaa}}{2-\aaaa}+\dddd{1}{2}x^{1-\aaaa}h+\zzzz(\aaaa-1)h^{2-\aaaa}+O\llll(h^2 \rrrr).$$
\end{lem}
\begin{proof} Consider the first terms of \eqref{Expansion1}
$$\sum_{k=1}^{n-1} k^{1-\aaaa}= \dddd{n^{2-\aaaa}}{2-\aaaa}-\dddd{n^{1-\aaaa}}{2}+\zzzz(\aaaa-1)+O\llll(\dddd{1}{n^a} \rrrr),$$
$$\sum_{k=1}^{n} k^{1-\aaaa}= \dddd{n^{2-\aaaa}}{2-\aaaa}+\dddd{n^{1-\aaaa}}{2}+\zzzz(\aaaa-1)+O\llll(\dddd{1}{n^a} \rrrr).$$
Multiply by $h^{2-\aaaa}$ 
$$h^{2-\aaaa}\sum_{k=1}^{n} k^{1-\aaaa} = \dddd{x^{2-\aaaa}}{2-\aaaa}+
\dddd{1}{2} x^{1-\aaaa} h+\zzzz(\aaaa-1)h^{2-\aaaa}+O\llll(\dddd{h^{2-\aaaa}}{n^\aaaa} \rrrr).$$
We have that $h^{2-\aaaa}/n^\aaaa=h^2/x^\aaaa$. Hence
$$\mathcal{L}^{(\aaaa)}_{1,h}=\dddd{x^{2-\aaaa}}{2-\aaaa}+\dddd{1}{2}x^{1-\aaaa}h+\zzzz(\aaaa-1)h^{2-\aaaa}+O\llll(h^2 \rrrr).$$
\end{proof}
In the next theorem we determine a second order approximation  for the left Riemann sums of the fractional integral $J^{2-\aaaa}y(x)$ when the function  $y(x)$ is a polynomial.
\begin{thm} Let $x=n h$ and  $y(x)$ be a polynomial.
\begin{equation}\label{A3}
\mathcal{L}_{y,h}^{(\aaaa)}=\GGGG(2-\aaaa)J^{2-\aaaa}y(x)+\dddd{y(0)}{2}x^{1-\aaaa}h+\zzzz(\aaaa-1)y(x)h^{2-\aaaa}+O\llll(h^2\rrrr).
\end{equation}
\end{thm}
\begin{proof}  Let $y(\xi)$ be a polynomial of degree $m$. The Taylor polynomial for $y(\xi)$ of degree $m$ at the point $\xi=x$ is equal to $y(\xi)$. 
$$y(\xi)=p_0+p_1(x-\xi)+\cdots+p_n(x-\xi)^m=p_0+\sum_{k=1}^m p_k(x-\xi)^k.$$
Denote
$$y_0(\xi)=y(\xi)-y(x)=\sum_{k=1}^m p_k(x-\xi)^k.$$
The function $(x-\xi)^{1-\aaaa}y_0(\xi)$ has a bounded derivative on the interval $[0,x]$. The trapezoidal approximation $\mathcal{T}_{y_0,h}^{(\aaaa)}$ is a second order approximation for the fractional integral
$\GGGG(2-\aaaa)J^{(2-\aaaa)}y_0(x)$. From Lemma 2 and \eqref{L1}
$$\mathcal{T}^{(\aaaa)}_{1,h}=\dddd{x^{2-\aaaa}}{2-\aaaa}+\zzzz(\aaaa-1)h^{2-\aaaa}+O\llll(h^2 \rrrr).$$
Then
  $$\mathcal{T}_{y,h}^{(\aaaa)}=\mathcal{T}_{y_0,h}^{(\aaaa)}+p_0\mathcal{T}_{1,h}^{(\aaaa)}=\GGGG(2-\aaaa)J^{2-\aaaa}y_0(x)+\dddd{p_0}{2-\aaaa}x^{2-\aaaa}+p_0\zzzz(\aaaa-1)h^{2-\aaaa}+O\llll(h^2 \rrrr).$$
We have that $y(x)=p_0$ and $J^{2-\aaaa}1=x^{2-\aaaa}/\GGGG(3-\aaaa)$,
$$\GGGG(2-\aaaa)J^{2-\aaaa}y(x)=\GGGG(2-\aaaa)J^{2-a}y_0(x)(x)+\dddd{p_0 }{2-\aaaa}x^{2-\aaaa}.$$
Hence
 $$\mathcal{T}^{(\aaaa)}_{y,h}=\mathcal{L}^{(\aaaa)}_{y,h}-\dddd{y(0)}{2}x^{1-\aaaa}h=\GGGG(2-\aaaa)J^{2-\aaaa}y(x)+y(x)\zzzz(\aaaa-1)h^{2-\aaaa}+O\llll(h^2 \rrrr).$$
\end{proof}
In Theorem 3 we showed that \eqref{A3} is a second order approximation for the left Riemann sums and the fractional integral of order $2-\aaaa$, when the function $y(x)$ is a polynomial. From the Weierstrass Approximation Theorem every sufficiently differentiable function and its derivatives on the interval $[0,x]$ are uniform limit of polynomials. The class of functions for which Theorem 3 holds includes functions with bounded derivatives. In section 4, we present a proof for the second order approximation \eqref{A2} of the Caputo derivative.
\begin{table}[ht]
    \caption{Error and order of approximation \eqref{A3}  for   $y(x)=\cos x$ (left) and $y(x)=\ln (x+1)$ (right) on the interval $[0,1]$, when $\alpha=0.4$.}
    \begin{subtable}{0.5\linewidth}
      \centering
  \begin{tabular}{l c c }
  \hline \hline
    $h$ & $Error$ & $Order$  \\ 
		\hline \hline
$0.05$         &$0.00011853$            &$1.95822$\\
$0.025$        &$0.00003019$            &$1.97331$\\
$0.0125$       &$7.63\times10^{-6}$     &$1.98275$\\
$0.00625$      &$1.92\times10^{-6}$     &$1.98876$\\
$0.003125$     &$4.83\times10^{-7}$     &$1.99264$\\
		\hline
  \end{tabular}
    \end{subtable}%
    \begin{subtable}{.5\linewidth}
      \centering
				\quad
  \begin{tabular}{ l  c  c }
    \hline \hline
    $h$ & $Error$ &$Order$  \\ \hline \hline
$0.05$     & $0.00020451$    & $1.98580$ \\
$0.025$    & $0.00005145$    & $1.99083$ \\
$0.0125$   & $0.00001292$   & $1.99400$ \\
$0.00625$  & $3.24\times10^{-6}$   & $ 1.99606$ \\
$0.003125$ & $ 8.11\times10^{-7}$   & $1.99740$ \\       
\hline
  \end{tabular}
    \end{subtable} 
\end{table}
\section{Second Order Approximation for the Caputo Derivative}
In this section we use approximation \eqref{A3} to determine a second order discretization for the Caputo derivative of order $\aaaa$, by modifying the first three coefficients of approximation \eqref{A1} with the value of the Riemann zeta function at the point $\aaaa-1$.

Denote by $\Delta_h^1 y_n$ and $\Delta_h^2 y_n$ the forward difference and the central difference of the function $y(x)$ at the point $x_n=nh$.
$$\Delta_h^1 y_n=y_{n+1}-y_n,$$
$$\Delta_h^2 y_n=y_{n+1}-2y_n+y_{n-1}.$$
When $y(x)$ is a sufficiently differentiable function
$$y_{n+0.5}^{\prime}=\dddd{\Delta_h^1 y_n}{h}+O\llll(h^2\rrrr),\quad 
y_{n}^{\prime\prime}=\dddd{\Delta_h^2 y_n}{h^2}+O\llll(h^2\rrrr).
$$
\begin{lem}  
$$\mathcal{A}_{h} y_{n}=\sum_{k=1}^{n-1}k^{1-\aaaa}\Delta_h^2  y_{n-k}+n^{1-\aaaa}\Delta_h^1 y_0.$$
\end{lem}
\begin{proof} 
\begin{align*}
\mathcal{A}_{h}& y_{n}=\sum_{k=0}^n\ssss_k^{(\aaaa)}y_{n-k}=\ssss_0^{(\aaaa)}y_{n}+\sum_{k=1}^{n-1}\ssss_k^{(\aaaa)}y_{n-k}+\ssss_n^{(\aaaa)}y_{0}\\
=&y_{n}+\sum_{k=1}^{n-1}\llll((k-1)^{1-\aaaa}-2k^{1-\aaaa}+(k+1)^{1-\aaaa}\rrrr)y_{n-k}+\ssss_n^{(\aaaa)}y_{0}\\
=&y_{n}+\sum_{k=1}^{n-1}(k-1)^{1-\aaaa}y_{n-k}-2\sum_{k=1}^{n-1}k^{1-\aaaa}y_{n-k}+\sum_{k=1}^{n-1}(k+1)^{1-\aaaa}y_{n-k}+\ssss_n^{(\aaaa)}y_{0}.
\end{align*}
Substitute $K=k-1$ in the first sum and $K=k+1$ in the third sum
\begin{align*}
\mathcal{A}_{h} y_{n}=y_{n}+\sum_{K=1}^{n-2} K^{1-\aaaa}y_{n-K-1}-2\sum_{k=1}^{n-1}k^{1-\aaaa}y_{n-k}+\sum_{K=2}^{n}K^{1-\aaaa}y_{n-K+1}+\ssss_n^{(\aaaa)}y_{0}.
\end{align*}
We have that
$$\sum_{k=1}^{n-2} k^{1-\aaaa}y_{n-k-1}=\sum_{k=1}^{n-1} k^{1-\aaaa}y_{n-k-1}-(n-1)^{1-\aaaa}y_0,$$
$$y_n+\sum_{k=2}^{n} k^{1-\aaaa}y_{n-k+1}=\sum_{k=1}^{n} k^{1-\aaaa}y_{n-k+1}=\sum_{k=1}^{n-1} k^{1-\aaaa}y_{n-k+1}+n^{1-\aaaa}y_1.$$
Then
$$\mathcal{A}_{h} y_{n}=\sum_{k=1}^{n-1} k^{1-\aaaa}\llll(y_{n-k+1}-2y_{n-k}+y_{n-k-1}\rrrr)+n^{1-\aaaa}\llll(y_1-y_0\rrrr),$$
because $\ssss_n^{(\aaaa)}=(n-1)^{1-\aaaa}-n^{1-\aaaa}$.
\end{proof}
\begin{lem} Suppose that $y(x)$ is sufficiently differentiable function on $[0,nh]$ 
$$\dddd{1}{h^\aaaa}\mathcal{A}_{h} y_{n}=h^{2-\aaaa}\sum_{k=1}^{n-1}k^{1-\aaaa}  y^{\prime\prime}_{n-k}+(nh)^{1-\aaaa} y^\prime_{0.5}+O\llll(h^2\rrrr).$$
\end{lem}
\begin{proof} From Lemma 4
$$\dddd{1}{h^\aaaa}\mathcal{A}_{h} y_{n}=\sum_{k=1}^{n-1}k^{1-\aaaa}\Delta_h^2  y_{n-k}+n^{1-\aaaa}\Delta_h^1 y_0=h^{2-\aaaa}\sum_{k=1}^{n-1}k^{1-\aaaa}\dddd{\Delta_h^2}{h^2}  y_{n-k}+n^{1-\aaaa}h^{1-\aaaa}\dddd{\Delta_h^1 y_0}{h},$$
$$\dddd{1}{h^\aaaa}\mathcal{A}_{h} y_{n}=h^{2-\aaaa}\sum_{k=1}^{n-1}k^{1-\aaaa}\llll(  y^{\prime\prime}_{n-k}+O\llll(h^2\rrrr)\rrrr)+
(nh)^{1-\aaaa}\llll(  y^{\prime}_{0.5}+O\llll(h^2\rrrr)\rrrr),$$
$$\dddd{1}{h^\aaaa}\mathcal{A}_{h} y_{n}=h^{2-\aaaa}\sum_{k=1}^{n-1}k^{1-\aaaa}  y^{\prime\prime}_{n-k}+
(n h)^{1-\aaaa}  y^{\prime}_{0.5}+O\llll(h^2\rrrr)\llll((n h)^{1-\aaaa}+h^{2-\aaaa}\sum_{k=1}^{n-1}k^{1-\aaaa}  \rrrr).$$
The number $(n h)^{1-\aaaa}\sim O(1)$ is bounded. From  \eqref{Expansion1} we have
$$h^{2-\aaaa}\sum_{k=1}^{n-1}k^{1-\aaaa}\sim h^{2-\aaaa}O\llll(n^{2-\aaaa}\rrrr)\sim O(1).$$
Therefore
$$\dddd{1}{h^\aaaa}\mathcal{A}_{h} y_{n}=h^{2-\aaaa}\sum_{k=1}^{n-1}k^{1-\aaaa}  y^{\prime\prime}_{n-k}+
(n h)^{1-\aaaa}  y^{\prime}_{0.5}+O\llll(h^2\rrrr).$$
\end{proof}
\begin{thm} Let $y$ be a polynomial and $x=nh$.
$$\dddd{1}{h^\aaaa}\mathcal{A}_{h} y(x)=\GGGG(2-\aaaa)y^{(\aaaa)}_n+\zzzz(\aaaa-1)y^{\prime\prime}(x)h^{2-\aaaa}+O\llll(h^2\rrrr).
$$
\begin{proof} From Lemma 5
\begin{align*}
\dddd{1}{h^\aaaa}\mathcal{A}_{h} y(x)=&h^{2-\aaaa}\sum_{k=1}^{n-1}k^{1-\aaaa}  y^{\prime\prime}_{n-k}+(nh)^{1-\aaaa} y^\prime_{0.5}+O\llll(h^2\rrrr)=\\
&h^{2-\aaaa}\sum_{k=1}^{n}k^{1-\aaaa}  y^{\prime\prime}_{n-k}-h^{2-\aaaa}n^{1-\aaaa}y_0^{\prime\prime}+x^{1-\aaaa} y^\prime_{0.5}+O\llll(h^2\rrrr).
\end{align*}
Then
$$\dddd{1}{h^\aaaa}\mathcal{A}_{h} y(x)=\mathcal{L}^{(\aaaa)}_{y^{\prime\prime},h}-
x^{1-\aaaa}\llll(h y_0^{\prime\prime}-y^\prime_{0.5}\rrrr)+O\llll(h^2\rrrr).$$
From Claim 1 and Theorem 3
$$\mathcal{L}^{(\aaaa)}_{y^{\prime\prime},h}=\GGGG(2-\aaaa)J^{2-\aaaa}y^{\prime\prime}(x)+\dddd{y^{\prime\prime}(0)}{2}x^{1-\aaaa}h+\zzzz(\aaaa-1)y^{\prime\prime}(x)h^{2-\aaaa}+O\llll(h^2\rrrr),$$
$$\GGGG(2-\aaaa)y^{(\aaaa)}(x)=\GGGG(2-\aaaa)J^{2-\aaaa}y''(x)+y'(0)x^{1-\aaaa}.$$ 
Then
$$\mathcal{L}^{(\aaaa)}_{y^{\prime\prime},h}=\GGGG(2-\aaaa)y^{(\aaaa)}(x)+\zzzz(\aaaa-1)y^{\prime\prime}(x)h^{2-\aaaa}-x^{1-\aaaa}\llll(y'_0-\dddd{y^{\prime\prime}_0 h}{2}\rrrr)+O\llll(h^2\rrrr),$$
$$\dddd{1}{h^\aaaa}\mathcal{A}_{h} y(x)=\GGGG(2-\aaaa)y^{(\aaaa)}(x)+\zzzz(\aaaa-1)y^{\prime\prime}(x)h^{2-\aaaa}-x^{1-\aaaa}\llll(y'_0+\dddd{y^{\prime\prime}_0 h}{2}-y^\prime_{0.5}\rrrr)+O\llll(h^2\rrrr).$$
By Taylor's expansion
$$y'_0+\dddd{y^{\prime\prime}_0 h}{2}-y^\prime_{0.5}=O\llll(h^2\rrrr).$$
Hence
$$\dddd{1}{h^\aaaa}\mathcal{A}_{h} y(x)=\GGGG(2-\aaaa)y^{(\aaaa)}(x)+\zzzz(\aaaa-1)y^{\prime\prime}(x)h^{2-\aaaa}+O\llll(h^2\rrrr).$$
\end{proof}
\end{thm}
In Theorem 6 we determined the second order approximation for the Caputo derivative
\begin{equation}\label{A5}
y^{(\aaaa)}_n=\dddd{1}{\GGGG(2-\aaaa)h^\aaaa}\mathcal{A}_{h} y_n-\dddd{\zzzz(\aaaa-1)}{\GGGG(2-\aaaa)}y^{\prime\prime}_n h^{2-\aaaa}+O\llll(h^2\rrrr).
\end{equation}
\begin{cor} Let $y(x)$ be a polynomial.
\begin{equation*}
y^{(\aaaa)}_n=\dddd{1}{\GGGG(2-\aaaa)h^\aaaa}\sum_{k=0}^n \delta_k^{(\aaaa)} y_{n-k}+O\llll(h^{2}\rrrr),
\end{equation*}
where $\dddddd_k^{(\aaaa)}=\ssss_k^{(\aaaa)}$ for $2\leq k\leq n$ and
$$\dddddd_0^{(\aaaa)}=\ssss_0^{(\aaaa)}-\zzzz(\aaaa-1),\; \dddddd_1^{(\aaaa)}=\ssss_1^{(\aaaa)}+2\zzzz(\aaaa-1),\; \dddddd_2^{(\aaaa)}=\ssss_2^{(\aaaa)}-\zzzz(\aaaa-1).$$
\end{cor}
\begin{proof} The second order backward difference approximation for the second derivative $y^{\prime\prime}_n$ has accuracy $O(h)$.
$$y^{\prime\prime}_n=\dddd{y_{n}-2y_{n-1}+y_{n-2}}{h^2}+O(h).$$
From approximation \eqref{A5}
$$y^{(\aaaa)}_n=\dddd{1}{\GGGG(2-\aaaa)h^\aaaa}\mathcal{A}_{h} y_n-\dddd{\zzzz(\aaaa-1)}{\GGGG(2-\aaaa)}\llll(\dddd{y_{n}-2y_{n-1}+y_{n-2}}{h^2}+O(h)\rrrr) h^{2-\aaaa}+O\llll(h^2\rrrr),$$
%$$\sum_{k=0}^{n} \ssss_k^{(\alpha)} y_{n-k}$$
$$y^{(\aaaa)}_n=\dddd{1}{\GGGG(2-\aaaa)h^\aaaa}\mathcal{A}_{h} y_n-\dddd{\zzzz(\aaaa-1)}{\GGGG(2-\aaaa) h^{\aaaa}}\llll(y_{n}-2y_{n-1}+y_{n-2}\rrrr)+O\llll(h^2\rrrr),$$
$$y^{(\aaaa)}_n=\dddd{1}{\GGGG(2-\aaaa)h^\aaaa}\llll(\sum_{k=0}^{n} \ssss_k^{(\alpha)} y_{n-k}-\zzzz(\aaaa-1)\llll(y_{n}-2y_{n-1}+y_{n-2}\rrrr)\rrrr)+O\llll(h^2\rrrr).$$
\end{proof}
\begin{table}[ht]
    \caption{Error and order of  approximation \eqref{A2} for  Caputo derivative of order  $\alpha=0.25$ and $y(x)=\cos x$ (left),  $\;y(x)=\ln (x+1)$ (right) on  $[0,1]$. }
    \begin{subtable}{0.5\linewidth}
      \centering
  \begin{tabular}{l c c }
  \hline \hline
    $h$ & $Error$ & $Order$  \\ 
		\hline \hline
$0.05$         &$0.000081955$           &$2.30047$\\
$0.025$        &$0.000017556$     &$2.22284$\\
$0.0125$       &$3.95\times10^{-6}$     &$2.15376$\\
$0.00625$      &$9.20\times10^{-7}$     &$2.10073$\\
$0.003125$     &$2.20\times10^{-7}$     &$2.06368$\\
		\hline
  \end{tabular}
    \end{subtable}%
    \begin{subtable}{.5\linewidth}
      \centering
				\quad
  \begin{tabular}{ l  c  c }
    \hline \hline
    $h$ & $Error$ &$Order$  \\ \hline \hline
$0.05$         &$0.000029455$          &$2.31171$\\
$0.025$        &$6.39\times10^{-6}$     &$2.20475$\\
$0.0125$       &$1.46\times10^{-6}$     &$2.13162$\\
$0.00625$      &$3.44\times10^{-7}$     &$2.08272$\\
$0.003125$     &$8.31\times10^{-8}$     &$2.05103$\\      
\hline
  \end{tabular}
    \end{subtable} 
\end{table}
Denote
$$\mathcal{B}_h y(x)=\sum_{k=0}^n \delta_k^{(\aaaa)} y(x-kh).$$
In Corollary 7 we showed that \eqref{A2} is a second order approximation for the Caputo derivative of polynomials. Now we use the Weierstrass Approximation Theorem to extend the result to differentiable functions.
\begin{thm} Let $x=nh$ and $y$ be a sufficiently differentiable function.
$$y^{(\aaaa)}(x)=\dddd{1}{\GGGG(2-\aaaa)h^\aaaa}\mathcal{B}_h y(x)+O\llll(h^{2}\rrrr).$$
\end{thm}
\begin{proof} By the Weierstrass Approximation Theorem every continuous function is a uniform limit of polynomials. Let $\eeee>0$ and $p_\eeee(x)$ be a polynomial such that
$$\llll|y'(t)-p_\eeee(t)\rrrr|<\eeee,$$
for all $t\in [0,x]$. Define
$$q_\eeee(t)=y(0)+\int_0^t p_\eeee(\xi)d\xi.$$
The function $q_\eeee(t)$ is polynomial, and $q^\prime_\eeee(t)= p^\prime_\eeee(t)$. We have that
$$\llll|y(t)-q_\eeee(t)\rrrr|=\llll|\int_0^t \llll(y^\prime(\xi)-p_\eeee(\xi) \rrrr)d\xi\rrrr|\leq
\int_0^t \llll|y^\prime(\xi)-p_\eeee(\xi)\rrrr| d\xi<\int_0^x \eeee d\xi\leq x\eeee,$$
for all $t\in [0,x]$.
$$\bullet\;\Gamma (1-\aaaa)\llll|y^{(\aaaa)}(x)-q_\eeee^{(\aaaa)}(x)\rrrr|=\llll|\int_0^x \dddd{y^\prime(\xi)-q^\prime_\eeee(\xi)}{(x-\xi)^{\aaaa}}d\xi\rrrr|\leq \int_0^x \dddd{\llll|y^\prime(\xi)-p_\eeee(\xi)\rrrr|}{(x-\xi)^{\aaaa}}d\xi,$$
$$\llll|y^{(\aaaa)}(x)-q_\eeee^{(\aaaa)}(x)\rrrr|<
\dddd{\eeee}{\Gamma (1-\aaaa)}\int_0^x (x-\xi)^{-\aaaa}d\xi=\dddd{\eeee x^{1-\aaaa}}{\Gamma (2-\aaaa)}.$$
Therefore
$$\lim_{\eeee \rrrrrrrr 0}{q_\eeee^{(\aaaa)}(x)}=y^{(\aaaa)}(x).
$$
$\bullet\;$ Now we estimate $\mathcal{B}_h \llll( y(x)-q_\eeee (x)\rrrr)$.
\begin{align*}
\llll|\sum_{k=0}^n \delta_k^{(\aaaa)} \llll( y_{n-k}-q_{\eeee,n-k}\rrrr)\rrrr|\leq \sum_{k=0}^n |\delta_k^{(\aaaa)}| \llll| y_{n-k}-q_{\eeee,n-k}\rrrr|\leq x\eeee \sum_{k=0}^n |\delta_k^{(\aaaa)}|.
\end{align*}
We have that
$$\sum_{k=0}^n |\delta_k^{(\aaaa)}|\leq \sum_{k=0}^n |\ssss_k^{(\aaaa)}|+3|\zzzz(\aaaa-1)|=2-3\zzzz(\aaaa-1).$$
Hence
\begin{align*}
\llll|\mathcal{B}_h \llll( y(x)-q_\eeee (x)\rrrr)\rrrr|\leq (2-3\zzzz(\aaaa-1))x\eeee ,
\end{align*}
and
$$\lim_{\eeee \rrrrrrrr 0} \mathcal{B}_h q_\eeee (x) =\mathcal{B}_h  y(x).$$
$\bullet\;$From Corollary 7
$$q_\eeee^{(\aaaa)}(x)=\dddd{1}{\GGGG(2-\aaaa)h^\aaaa}\mathcal{B}_h q_\eeee(x)+O\llll(h^{2}\rrrr).$$
By letting $\eeee \rrrrrrrr 0$, we obtain
$$y^{(\aaaa)}(x)=\dddd{1}{\GGGG(2-\aaaa)h^\aaaa}\mathcal{B}_h y(x)+O\llll(h^{2}\rrrr).$$
\end{proof}
\begin{figure}[ht]
\centering
\begin{minipage}{.47\textwidth}
  \centering
  \includegraphics[width=.98\linewidth]{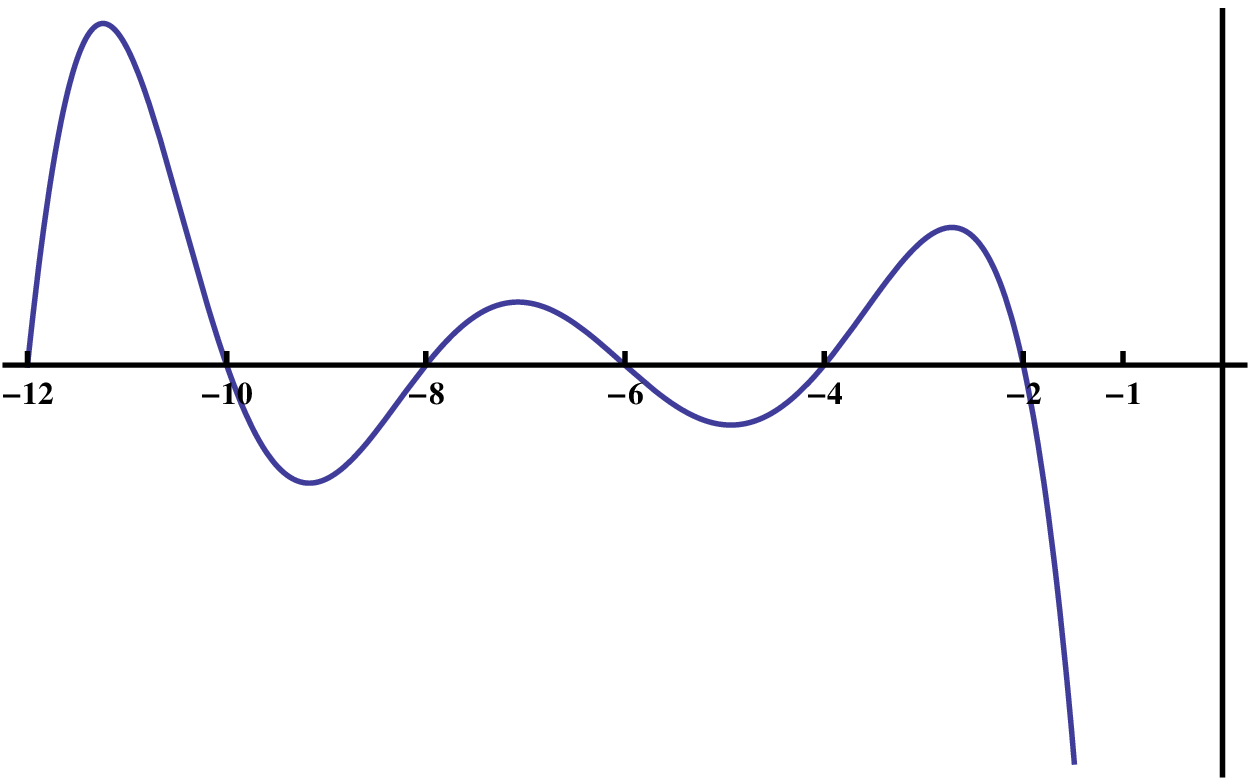}
  \label{fig:test3}
\end{minipage}%
\hspace{0.3cm}
\begin{minipage}{.47\textwidth}
  \centering
  \includegraphics[width=.98\linewidth]{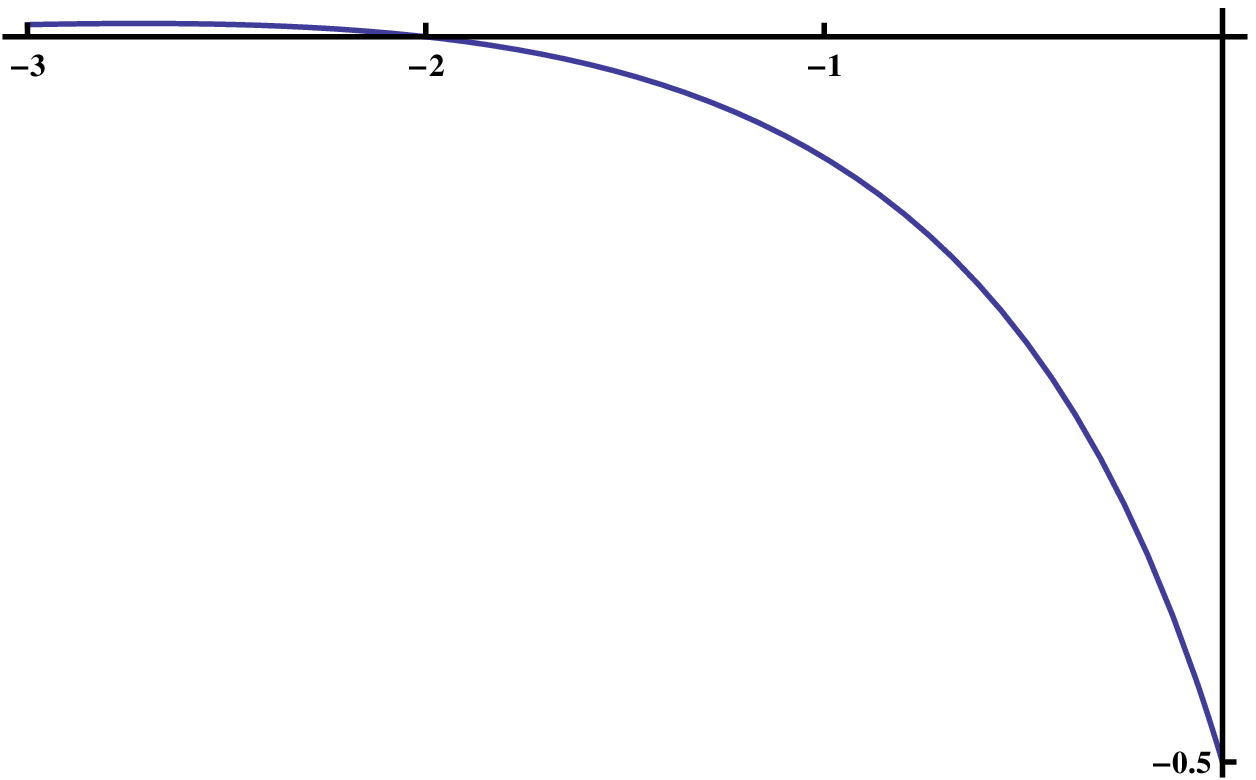}
  \label{fig:test4}
\end{minipage}
 \captionof{figure}{Graphs of the Riemann zeta function on the intervals $[-12,-1]$ and $[-3,0]$.}
\end{figure}
In Table 3 we compute the error and the numerical order of approximation \eqref{A2} for the Caputo derivative of the functions $y(x)=\cos x$ and $y(x)=\ln (x+1)$ on the interval $[0,1]$, when $\aaaa=0.25$. In Claim 9 and Lemma 10 we discuss the properties of the coefficients $\ssss_2^{(\aaaa)}$ and $\dddddd_2^{(\aaaa)}$.
\begin{clm} Let $0<\aaaa<1$
$$-0.1<\ssss_2^{(\aaaa)}<0.$$
\end{clm}
\begin{proof} Denote 
$$\ssss(\aaaa)=-\ssss_2^{(1-\aaaa)}=2^{\aaaa+1}-3^\aaaa-1.$$
The function $\ssss(\aaaa)$ has values $\ssss(0)=\ssss(1)=0$, and
$$\ssss^\prime (\aaaa)=\ln 2. 2^{\aaaa+1}-\ln 3. 3^\aaaa.$$
The first derivative of $\ssss(\aaaa)$ is zero when
$$\ln 2. 2^{\aaaa+1}=\ln. 3 3^\aaaa,\quad \llll(\dddd{3}{2}\rrrr)^\aaaa=\dddd{2.\ln 2}{\ln 3},\quad \aaaa=\dddd{\ln\llll( \dddd{2.\ln 2}{\ln 3}\rrrr)}{\ln (3/2)}\approx 0.5736.
$$
The function $\ssss(\aaaa)$ is positive and has a maximum value $\ssss(0.5736)\approx0.0985$ on the interval $[0,1]$.
\end{proof}
The Riemann zeta function has zeroes at the negative even integers and is decreasing on the interval $[-2,1]$. The value of $\zzzz(\aaaa-1)$ is negative, when $\aaaa$ is between $0$ and $1$. Then $\dddddd_0^{(\aaaa)}>0$ and $\dddddd_1^{(\aaaa)}<0$. From the properties of the coefficients of \eqref{A1}, the numbers $\dddddd_n^{(\aaaa)}=\ssss_n^{(\aaaa)}$ are negative, for $n\geq 3$. 
\begin{lem} The number $\dddddd_2^{(\aaaa)}$ is positive when $0<\aaaa<1$.
\end{lem}
\begin{proof} From the definition of $\dddddd_2^{(1-\aaaa)}$
$$\dddddd_2^{(1-\aaaa)}=\ssss_2^{(1-\aaaa)}-\zzzz(-\aaaa)=-\ssss(\aaaa)-\zzzz(-\aaaa)=z(\aaaa)-\ssss(\aaaa).$$
where $z(\aaaa)=-\zzzz (-\aaaa)$. The function $z(\aaaa)$ is decreasing on the interval $[0,1]$ with values at the endpoints  $z(0)=0.5$ and $z(1)=1/12=0.08333$. The function $\ssss(\aaaa)$ is increasing on the interval $[0,0.5736]$ and decreasing on the interval $[0.5736,1]$.

Now we show that the minimum values of  $\dddddd_2^{(1-\aaaa)}$  on the intervals $[0,0.8]$ and $[0.8,1]$ are positive.
$$\min_{\aaaa \in [0,0.8]} \dddddd_2^{(1-\aaaa)}>\min_{\aaaa \in [0,0.8]} z(\aaaa)-\max_{\aaaa \in [0,0.8]} \ssss(\aaaa),$$
$$\min_{\aaaa \in [0,0.8]} \dddddd_2^{(1-\aaaa)}>z(0.8)-\ssss(0.5736)\approx 0.122-0.0985=0.0235.$$
and
$$\min_{\aaaa \in [0.8,1]} \dddddd_2^{(1-\aaaa)}>\min_{\aaaa \in [0.8,1]} z(\aaaa)-\max_{\aaaa \in [0.8,1]} \ssss(\aaaa),$$
$$\min_{\aaaa \in [0.8,1]} \dddddd_2^{(1-\aaaa)}>z(1)-\ssss(0.8)\approx 0.083-0.074=0.009.$$
Therefore the numbers $\dddddd_2^{(1-\aaaa)}$ are positive when $0<\aaaa< 1$.
\end{proof}
\section{Numerical Experiments}
In section 4 we showed that the approximation for the Caputo derivative 
$$y^{(\aaaa)}_n\approx\dddd{1}{\GGGG(2-\aaaa)h^\aaaa}\sum_{k=0}^n \delta_k^{(\aaaa)} y_{n},$$
has accuracy $O\llll( h^2\rrrr)$ when $n\geq2$.  The numbers $\dddddd_k^{(\alpha)}$ satisfy
\begin{equation*}
\dddddd_0^{(\alpha)}>0,\;\dddddd_1^{(\alpha)}<0,\;\dddddd_2^{(\alpha)}>0,\; \dddddd_3^{(\alpha)}<\dddddd_4^{(\alpha)}<\cdots<\dddddd_k^{(\alpha)}<\cdots<0,\; \sum_{k= 0}^\infty \dddddd_k^{(\alpha)} = 0.
\end{equation*}
In this section we compare the performance of the numerical solutions of the fractional relaxation and time-fractional subdiffusion equations using approximations \eqref{A1} and \eqref{A2} for Caputo derivative. From the Mean-Value theorem  for the Caputo derivative  
$$y(h)-y(0)= \dddd{h^\aaaa}{\GGGG(1+\aaaa)}y^{(\aaaa)}(\theta), \qquad (0<\theta<h).$$
The numbers  $\GGGG(1+\aaaa)$ and $\GGGG(2-\aaaa)$ are between $0$ and $1$, when $0<\aaaa<1$.
\begin{lem} Let $y$ be a sufficiently differentiable function on $[0,h]$
\begin{equation}\label{A4}
y(h)-y(0)-h^\aaaa \GGGG(2-\aaaa)y^{(\aaaa)}(h)=O\llll(h^2\rrrr).
\end{equation}
\end{lem} 
\begin{proof} From the definition of the Caputo derivative
$$y^{(\aaaa)}(h)=\dddd{1}{\Gamma (1-\aaaa)}\int_0^h \dfrac{y^\prime(\xi)}{(h-\xi)^{\aaaa}}d\xi.$$
Expand the function $y^\prime(\xi)$ around $\xi=0$
$$y^\prime(\xi)=y^\prime\llll(0 \rrrr)+\xi y^{\prime \prime}\llll(0\rrrr) +O(h^2),$$
\begin{align*}
& \Gamma (1-\aaaa)y^{(\aaaa)}(h)=\int_0^h \dfrac{y^\prime\llll(0 \rrrr)+\xi y^{\prime\prime}(0)+O\llll(h^2\rrrr)}{(h-\xi)^{\aaaa}}d\xi= \\
&y^\prime\llll(0 \rrrr)\int_0^h \dfrac{1}{(h-\xi)^{\aaaa}}+
y^{\prime\prime}(0)\int_0^{h} \dfrac{\xi}{(h-\xi)^{\aaaa}}d\xi+
O\llll(h^2\rrrr)\int_{0}^h\dfrac{1}{(h-\xi)^{\aaaa}}d\xi.
\end{align*}
We have that
$$\int_0^{h} \dfrac{1}{(h-\xi)^{\aaaa}}d\xi=\dfrac{h^{1-\aaaa}}{1-\aaaa},\quad
\int_0^{h} \dfrac{\xi}{(h-\xi)^{\aaaa}}d\xi=\dddd{h^{2-\aaaa}}{(1-\aaaa)(2-\aaaa)}.$$
Then
$$\Gamma (1-\aaaa)y^{(\aaaa)}(h)=\dfrac{h^{1-\aaaa}}{1-\aaaa}y^\prime (0 )+\dfrac{h^{2-\aaaa}}{(1-\aaaa)(2-\aaaa)}y^{\prime\prime} (0 )+O\llll(h^{3-\aaaa}\rrrr),$$
$$\Gamma (2-\aaaa)h^\aaaa y^{(\aaaa)}(h)=h\llll(y^\prime (0 )+\dfrac{h}{2-\aaaa}y^{\prime\prime} (0 )\rrrr)+O\llll(h^{3-\aaaa}\rrrr),$$
$$\Gamma (2-\aaaa)h^\aaaa y^{(\aaaa)}(h)=h\llll(y^\prime (0 )+\dfrac{h}{2}y^{\prime\prime} (0 )\rrrr)+O\llll(h^{2}\rrrr),$$
$$\Gamma (2-\aaaa)h^\aaaa y^{(\aaaa)}(h) =hy^\prime\llll(\dddd{h}{2}\rrrr)+O\llll(h^{2}\rrrr)= y (h)-y(0)+O\llll(h^{2}\rrrr).$$
\end{proof}
\subsection{Numerical Solution of the Fractional Relaxation Equation}
The fractional relaxation equation \eqref{freqn} is an ordinary fractional differential equation with constant coefficients. The exact solution of the fractional relaxation equation is determined with the Laplace transform method \cite{Podlubny1999}. Numerical solutions of the relaxation equation are discussed in [19-21].
In this section we compare the numerical solutions of the equation 
\begin{equation}\label{REQ}
y^{(\aaaa)}+y=F(t),
\end{equation}
for approximations \eqref{A1} and \eqref{A2} of the Caputo derivative. 
When the solution $y(t)$ of \eqref{REQ} is a continuously differentiable function, the initial condition  $y(0)$ is determined from the function $F(t)$ by $y(0)=F(0)$. Let
	$$F(t)= 1-4t+5t^2-\dddd{4}{\GGGG(2-\aaaa)}t^{1-\aaaa} +\dddd{10}{\GGGG(3-\aaaa)}t^{2-\aaaa}.$$
  Equation \eqref{REQ} has the solution
	$$y(t)= 1-4t+5t^2,$$ 
	and the initial value $y(0)=1$. Now we determine a second order numerical solution of \eqref{REQ} on the interval $[0,1]$, using approximations \eqref{A2}  and \eqref{A4} for the Caputo derivative. 
	
	Let $h=1/N$, where $N$ is a positive integer, and $y_n=y(x_n)=y(n h)$. In Lemma 11, we showed that \eqref{A4} 
	\begin{equation}\label{A6}
	\dddd{y(h)-y(0)}{\Gamma(2-\aaaa)h^\aaaa}=y^{(\aaaa)}(h)+O\llll(h^{2-\aaaa}\rrrr).
	\end{equation}
	Approximate the Caputo derivative $y^{(\aaaa)}(h)$ in equation \eqref{REQ}
	$$\dddd{y(h)-y(0)}{\Gamma(2-\aaaa)h^\aaaa}+y(h)=F(h)+O\llll(h^{2-\aaaa}\rrrr),$$
	$$y_1\llll(1+\Gamma(2-\aaaa)h^\aaaa\rrrr)=y_0+\Gamma(2-\aaaa)h^\aaaa F_1+O\llll(h^2\rrrr).$$
Let $\llll\{\tilde{y}_k\rrrr\}_{k=0}^N$ be an approximation for the exact solution $y_k$ at the points $x_k=k h$. Set $\tilde{y}_0=y(0)=1$. The value of $\tilde{y}_1$ is computed from the above approximation with accuracy $O\llll(h^{2}\rrrr)$
$$\tilde{y}_1=\dddd{\tilde{y}_0+\Gamma(2-\aaaa)h^\aaaa F_1}{1+\Gamma(2-\aaaa)h^\aaaa}.$$		
	The numbers $\tilde{y}_n$, for $n\geq 2$, are  computed from  equation \eqref{REQ} by approximating the Caputo derivative $y_n^{(\aaaa)}$ with \eqref{A2}.
		$$\dddd{1}{\GGGG(2-\aaaa)h^\aaaa}\sum_{k=0}^n \delta_k^{(\aaaa)} y_{n-k}+y_n=F_n+O\llll(h^2 \rrrr),$$
				$$y_n\llll(\dddddd_0^{(\aaaa)}+\Gamma(2-\aaaa)h^\aaaa\rrrr)=\Gamma(2-\aaaa)h^\aaaa F_n-\sum_{k=1}^{n} \delta_k^{(\aaaa)} y_{n-k}+O\llll(h^{2+\aaaa}\rrrr).$$
				The numerical solution $\llll\{\tilde{y}_k\rrrr\}_{k=0}^N$, for $2\leq n\leq N$, is computed explicitly with
			\begin{equation}\label{NS1}
				\tilde{y}_n=\dddd{1}{\dddddd_0^{(\aaaa)}+\GGGG(2-\aaaa)h^\aaaa}\llll(\GGGG(2-\aaaa)h^\aaaa F_n-\sum_{k=1}^{n} \delta_k^{(\aaaa)} \tilde{y}_{n-k} \rrrr).
\end{equation}
				Similarly, we obtain an explicit formula for the numerical solution $\llll\{\tilde{\tilde{y}}_k\rrrr\}_{k=0}^N$ of equation \eqref{REQ}, by approximating the Caputo derivative $y_n^{(\aaaa)}$ with \eqref{A1}
		\begin{equation}\label{NS2}
					\tilde{\tilde{y}}_n=\dddd{1}{1+\GGGG(2-\aaaa)h^\aaaa}\llll(\GGGG(2-\aaaa)h^\aaaa F_n-\sum_{k=1}^{n} \ssss_k^{(\aaaa)} \tilde{\tilde{y}}_{n-k} \rrrr).
					\end{equation}
		Numerical solution \eqref{NS1}  converges faster to the  solution of the fractional relaxation equation, because it has a second order accuracy $O\llll(h^{2} \rrrr)$, and the accuracy  of  numerical solution \eqref{NS2} is $O\llll(h^{2-\aaaa} \rrrr)$. 
		\begin{figure}[ht]
  \centering
  \caption{Graph of the exact solution of equation \eqref{REQ} and  numerical solutions \eqref{NS1}-black,  and \eqref{NS2}-red, for $h=0.1$ and $\alpha=0.8$.} % caption line
  \includegraphics[width=0.55\textwidth]{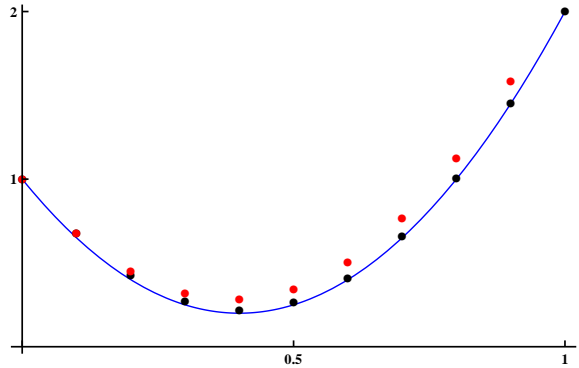}
\end{figure}
\begin{table}[ht]
    \caption{Maximum error and order of numerical solutions \eqref{NS2} and \eqref{NS1} for  equation \eqref{REQ} on the interval $[0,1]$, when $\alpha=0.8$. }
    \begin{subtable}{0.5\linewidth}
      \centering
  \begin{tabular}{l c c }
  \hline \hline
    $h$ & $Error$ & $Order$  \\ 
		\hline \hline
$0.05$         &$0.0628014$            &$1.17381$\\
$0.025$        &$0.0275997$            &$1.19262$\\
$0.0125$       &$0.0120751$     &$1.19262$\\
$0.00625$      &$0.0052704$     &$1.19603$\\
$0.003125$     &$0.0022975$     &$1.19785$\\
		\hline
  \end{tabular}
    \end{subtable}%
    \begin{subtable}{.5\linewidth}
      \centering
				\quad
  \begin{tabular}{ l  c  c }
    \hline \hline
    $h$ & $Error$ &$Order$  \\ \hline \hline
$0.05$     & $0.0081544$          & $1.85708$ \\
$0.025$    & $0.0021629$          & $1.91461$ \\
$0.0125$   & $0.0005599$          & $1.94979$ \\
$0.00625$  & $0.0001428$   & $ 1.97076$ \\
$0.003125$ & $0.0000361$        & $1.98307$ \\       
\hline
  \end{tabular}
    \end{subtable} 
\end{table}
\subsection{Numerical Solution of the Fractional Subdiffusion Equation}
The time-fractional fractional subdiffusion equation is obtained from the heat transfer equation by replacing the time derivative with a  fractional derivative of order $\aaaa$, where $0<\aaaa<1$
	\begin{equation}\label{TFSE}
\dfrac{\partial^\alpha u(x,t)}{\partial t^\alpha}=\dfrac{\partial^2 u(x,t)}{\partial x^2}+F(x,t),
\end{equation}
with initial and boundary conditions
$$u(x,0)=u_0(x),\;u(0,t)=u_L(t),\;u(1,t)=u_R(t).$$
Numerical solutions of the fractional subdiffusion equation are discussed in \cite{Dimitrov2013,GaoSunZhang2012,Lazarov2015_2,ZengLiLiuTurner2013}.
In this section we determine the numerical solutions \eqref{NS3} and \eqref{NS4} for the fractional subdiffusion equation  obtained by approximating the Caputo derivative with \eqref{A1} and \eqref{A2} on the region
	$$(x,t)\in [0,1]\times[0,1].$$
	
Let $h=1/N,\tttt=1/M$, where $M$ and $N$ are positive integers, and $\mathcal{G}$ be a grid on the square $[0,1]\times[0,1]$
$$\mathcal{G}=\llll\{(n h,m \tttt) \| 1\leq n\leq N, 1\leq m\leq M\rrrr\}.$$
Denote by $u_n^m$ and $F_n^m$ the values of the functions $u(x,t)$ and $F(x,t)$ on $\mathcal{G}$
$$u_n^m=u(n h,m \tttt),\quad F_n^m=F(n h,m\tttt).$$
By approximating the values of the Caputo derivative in the time direction at the points $(nh,\tttt)$ using  \eqref{A6} and using a central difference approximation for the second derivative in the space direction we obtain
$$\dddd{u_n^1-u_n^0}{\tttt^\aaaa\GGGG(2-\aaaa)}=
\dddd{u_{n-1}^1-2 u_n^1+u_{n+1}^1}{h^2}+F(nh,\tttt)+O\llll(h^2+\tttt^{2-\aaaa}\rrrr).$$
Let
$$\eta=\GGGG(2-\aaaa)\dddd{\tttt^\aaaa}{h^2}.$$
The solution of the fractional subdiffusion equation satisfies
$$-\eta u_{n-1}^1+(1+2\eta) u_n^1-\eta u_{n+1}^1=u_n^0+\GGGG(2-\aaaa)\tttt^\aaaa F(nh,\tttt)+O\llll( \tttt^\aaaa h^2+\tttt^2\rrrr).$$
Let $U_n^m$ be the numerical solution of the fractional subdiffusion equation on the grid $\mathcal{G}$. The numbers $U_n^m$ are approximations for the values of the solution $u_n^m=u(n h, m \tttt)$. The numbers $U_n^0$ are computed from the initial condition $U_n^0=u_0(n h)$. The numbers $U_n^1$ are approximations for the solution of \eqref{TFSE} at time $t=\tttt$. We compute the numbers $U_n^1$ implicitly from  the equations 
$$-\eta U_{n-1}^1+(1+2\eta) U_n^1-\eta U_{n+1}^1=U_n^0+\GGGG(2-\aaaa)\tttt^\aaaa F_n^1,$$
where the values of $U_0^1$ and $U_N^1$ are determined from the boundary conditions 
$$U_0^1=u_L(\tttt),\quad U_n^1=u_R(\tttt).$$
The numbers $U_n^1$ are computed with the  linear system $(k=2,\cdots,N-2)$
	\begin{equation*}\label{System}
	\left\{
	\begin{array}{l l}
(1+2\eta) U_1^1-\eta U_{2}^1=u_0(h)+\eta u_L(\tttt)+\GGGG(2-\aaaa)\tttt^\aaaa F_1^1\\
	-\eta U_{k-1}^1+(1+2\eta) U_k^1-\eta U_{k+1}^1=u_0(k h)+\GGGG(2-\aaaa)\tttt^\aaaa F_{k}^1  \\
	-\eta U_{N-2}^1+(1+2\eta) U_{N-1}^1=u_0((N-1)h)+\eta u_R(\tttt)+\GGGG(2-\aaaa)\tttt^\aaaa F_{N-1}^1.
	\end{array} 
		\right . 
	\end{equation*}
	Let $K$ be a tridiagonal matrix of dimension $N-1$ with values $ 1+2\eta$ on the main diagonal, and $-\eta$ on the diagonals above and below the main diagonal.
		$$K_{5} =
 \begin{pmatrix}
  1+2\eta &  -\eta    & 0        & 0          & 0       \\
 -\eta    &1+2\eta    & -\eta    & 0          & 0       \\
  0       & -\eta     & 1+2\eta  & -\eta      & 0       \\
  0       & 0         & -\eta    &  1+2\eta   & -\eta  \\
	0       & 0         & 0        & -\eta      &  1+2\eta   
 \end{pmatrix}
$$
	and $U^m=\llll(U_1^m,U_2^m,\cdots,U_{N-1}^m\rrrr)$. The vector $U^1$ is solution of the linear system
	\begin{equation}\label{G1}
	K U^1=R_1+\eta R_2,
	\end{equation}
	where $R_1$ and $R_2$ are the column vectors
	$$R_1=\llll[u_0(kh)+\GGGG(2-\aaaa)\tttt^\aaaa F_{n}^1 \rrrr]_{n=1}^{N-1},$$
	$$R_2=\llll[u_L(\tttt),0,\cdots,0,u_R(\tttt)\rrrr]^T.$$
	We determined a second order approximation $U^1$ for the solution of the fractional subdiffusion equation, on the first layer of $\mathcal{G}$, as a solution of the linear system \eqref{G1}.
	When $m\geq 2$ we discretize the Caputo derivative with  equation \eqref{TFSE} with the second order approximation \eqref{A2}
$$\dddd{1}{\tttt^\aaaa\GGGG(2-\aaaa)}\sum_{k=0}^m \delta_k^{(\aaaa)} u_n^{m-k}=
\dddd{u_{n-1}^m-2 u_n^m+u_{n+1}^m}{h^2}+F(nh,m\tttt)+O\llll(h^2+\tttt^2\rrrr).$$
The values of the numerical solution $U_n^m$ are determined from the equations
$$-\eta U_{n-1}^m+(\delta_0^{(\aaaa)}+2\eta) U_n^m-\eta U_{n+1}^m=-\sum_{k=1}^m \delta_k^{(\aaaa)} U_n^{m-k}+\GGGG(2-\aaaa)\tttt^\aaaa F_n^m,$$
and the boundary conditions
$$U_0^m=u_L(m\tttt),\quad U_N^m=u_R(m\tttt).$$
The vector $U^m$ is a solution of the linear system
\begin{equation}\label{NS3}
	(K-\zzzz(a-1)I)U^m=R_1+\eta R_2,
\end{equation}	
where $R_1$ and $R_2$ are the column vectors
	$$R_1=\llll[-\sum_{k=1}^m \delta_k^{(\aaaa)} U_n^{m-k}+\GGGG(2-\aaaa)\tttt^\aaaa F_{n}^m \rrrr]_{n=1}^{N-1},$$
	$$R_2=\llll[u_L(m\tttt),0,\cdots,0,u_R(m\tttt)\rrrr]^T.$$
	The numerical solution $\llll\{U^2,\cdots,U^M\rrrr\}$, using approximation \eqref{A2} for the Caputo derivative, is computed with linear systems \eqref{NS3}. Similarly we determine the numerical solution $\llll\{V^2,\cdots,V^M\rrrr\}$ for approximation \eqref{A1} with linear system \eqref{NS4} and first layer $V^1=U^1$.
	
	The numerical solution $V^m$ is computed with the linear system
	\begin{equation}\label{NS4}
	K V^m=R_1+\eta R_2
	\end{equation}
where $R_1$ and $R_2$ are the vectors
	$$R_1=\llll[-\sum_{k=1}^m \ssss_k^{(\aaaa)} V_n^{m-k}+\GGGG(2-\aaaa)\tttt^\aaaa F_{n}^m \rrrr]_{n=1}^{N-1},$$
	$$R_2=\llll[u_L(m\tttt),0,\cdots,0,u_R(m\tttt)\rrrr]^T.$$
	Numerical solution \eqref{NS3} has accuracy $O\llll(h^2 \rrrr)$ and the accuracy of \eqref{NS4} is $O\llll(h^{2-\aaaa} \rrrr)$. 
	 When 
	$$F(x,t)=2(1-3x)(5t^2-4t+1)+x^2(1-x)\llll(\dddd{10t^{2-\aaaa}}{\GGGG(3-\aaaa)}-\dddd{4t^{1-\aaaa}}{\GGGG(2-\aaaa)}\rrrr),$$
	the fractional sub-diffusion equation  
		\begin{equation}\label{TFSE2}
\dfrac{\partial^\alpha u(x,t)}{\partial t^\alpha}=\dfrac{\partial^2 u(x,t)}{\partial x^2}+F(x,t),\quad (x,t)\in [0,1]\times [0,1]
\end{equation}
with initial and boundary conditions
$$u(x,0)=x^2(1-x),\; u(0,t)=u(1,t)=0,$$
has solution
	$$u(x,t)=x^2(1-x)(1-4t+5t^2).$$
	The maximal error and numerical order of numerical solutions \eqref{NS4} and \eqref{NS3} for $\tttt=h$ and $\tttt=h/2$ at time $t=1$ for the fractional subdiffusion equation \eqref{TFSE2} are given in Table 5 and Table 6.\\
	
	\begin{table}[ht]
    \caption{Maximum error and order of numerical solutions \eqref{NS4} and \eqref{NS3}  for  equation \eqref{TFSE2} when $\alpha=0.6$ and $\tttt=h$, at time $t=1$. }
    \begin{subtable}{0.5\linewidth}
      \centering
  \begin{tabular}{l c c }
  \hline \hline
    $h\;(\tttt=h)$ & $Error$ & $Order$  \\ 
		\hline \hline
$0.05\quad$         &$0.00051794$            &$1.37686$\\
$0.025\quad$        &$0.00019766$            &$1.38974$\\
$0.0125\quad$       &$0.00007530$     &$1.39222$\\
$0.00625\quad$      &$0.00002864$     &$1.39467 $\\
$0.003125\quad$     &$0.00001087$     &$1.39657$\\
		\hline
  \end{tabular}
    \end{subtable}
    \begin{subtable}{.5\linewidth}
      \centering
				\quad\quad
  \begin{tabular}{ l  c  c }
    \hline \hline
    $\tttt\;(\tttt=h)$ & $Error$ &$Order$  \\ \hline \hline
$0.05\quad$     & $0.00001170$          & $1.93892$ \\
$0.025\quad$    & $2.99\times 10^{-6}$          & $1.96593$ \\
$0.0125\quad$   & $7.62\times 10^{-7}$          & $1.97559$ \\
$0.00625\quad$  & $1.93\times 10^{-7}$          & $1.98175$ \\
$0.003125\quad$ & $4.87\times 10^{-8}$        & $1.98648$ \\%  
\hline
  \end{tabular}
    \end{subtable} 
\end{table}
\begin{table}[t]
    \caption{Maximum error and order of numerical solutions \eqref{NS4} and \eqref{NS3}  for  equation \eqref{TFSE2} when $\alpha=0.4$ and $\tttt=0.5h$, at time $t=1$. }
    \begin{subtable}{0.5\linewidth}
      \centering
  \begin{tabular}{l c c }
  \hline \hline
    $h \;(h=2\tttt)$ & $Error$ & $Order$  \\ 
		\hline \hline
$0.05$         &$0.00006172$     &$1.56262$\\
$0.025$        &$0.00002069$     &$1.57697$\\
$0.0125$       &$6.91\times 10^{-6}$     &$1.58139$\\
$0.00625$      &$2.30\times10^{-6}$     &$1.58597$\\
$0.003125$     &$7.65\times10^{-7}$     &$1.58946$\\
		\hline
  \end{tabular}
    \end{subtable}%
    \begin{subtable}{.5\linewidth}
      \centering
				\quad\quad
  \begin{tabular}{ l  c  c }
    \hline \hline
    $h\;(h=2\tttt)$ & $Error$ &$Order$  \\ \hline \hline
$0.025$     & $4.41\times 10^{-6}$          & $1.98507$ \\
$0.0125$    & $1.11\times 10^{-6}$          & $1.99521$ \\
$0.00625$   & $2.78\times 10^{-7}$          & $1.99589$ \\
$0.00625$  & $6.95\times 10^{-8}$          & $ 1.99701$ \\
$0.003125$ & $1.74\times 10^{-8}$        & $1.99811$ \\%      
\hline
  \end{tabular}
    \end{subtable} 
\end{table}
In numerical solutions \eqref{NS3} and \eqref{NS4}, we use \eqref{A6} to obtain a second order approximation for the solution of the fractional subdiffusion equation on the first layer of $\mathcal{G}$. Another way to determine a second order approximation for the solution at time $t=\tttt$ is to compute the partial derivative $u_t(x,t)$ at time $t=0$ and approximate the solution $u(x,\tttt)$ with a second order Taylor expansion. The function $u(x,t)$ satisfies
	\begin{equation*}
\dfrac{\partial^\alpha u(x,t)}{\partial t^\alpha}=\dfrac{\partial^2 u(x,t)}{\partial x^2}+F(x,t).
\end{equation*}
Apply fractional differentiation of order $1-a$
	\begin{equation*}
\dfrac{\partial^{1-\aaaa}}{\partial t^{1-\alpha}}\dfrac{\partial^\alpha u(x,t)}{\partial t^\alpha}=\dfrac{\partial^{1-\aaaa}}{\partial t^{1-\alpha}}\dfrac{\partial^2 u(x,t)}{\partial x^2}+\dfrac{\partial^{1-\aaaa}F(x,t)}{\partial t^{1-\alpha}},
\end{equation*}
	\begin{equation*}
u_t(x,t)=\dfrac{\partial^{1-\aaaa}}{\partial t^{1-\alpha}}\dfrac{\partial^2 u(x,t)}{\partial x^2}+\dfrac{\partial^{1-\aaaa}F(x,t)}{\partial t^{1-\alpha}}.
\end{equation*}
Set $t=0$
\begin{equation*}
u_t(x,0)=\llll. \dfrac{\partial^{1-\aaaa}}{\partial t^{1-\alpha}}\dfrac{\partial^2 u(x,t)}{\partial x^2}\rrrr|_{t=0}+\llll.\dfrac{\partial^{1-\aaaa}F(x,t)}{\partial t^{1-\alpha}}\rrrr|_{t=0}.
\end{equation*}
When the solution $u(x,t)$ is a sufficiently smooth function
$$\llll. \dfrac{\partial^{1-\aaaa}}{\partial t^{1-\alpha}}\dfrac{\partial^2 u(x,t)}{\partial x^2}\rrrr|_{t=0}=0,$$
we obtain
\begin{equation*}
u_t(x,0)=\llll.\dfrac{\partial^{1-\aaaa}F(x,t)}{\partial t^{1-\alpha}}\rrrr|_{t=0}.
\end{equation*}
The values of the solution at time $t=\tttt$ are approximated using the second order Taylor expansion
$$u(x,\tttt)=u(x,0)+\tttt u_t(x,0)+O\llll(\tttt^2\rrrr)$$
In the fractional subdiffusion equation \eqref{TFSE2}
	$$F(x,t)=2(1-3x)(5t^2-4t+1)+x^2(1-x)\llll(\dddd{10t^{2-\aaaa}}{\GGGG(3-\aaaa)}-\dddd{4t^{1-\aaaa}}{\GGGG(2-\aaaa)}\rrrr),$$
$$	\dfrac{\partial^{1-\aaaa}F(x,t)}{\partial t^{1-\alpha}}=
2(1-3x)\llll(\dddd{10t^{1+\aaaa}}{\GGGG(2+\aaaa)}-\dddd{4t^\aaaa}{\GGGG(1+\aaaa)}\rrrr)+x^2(1-x)\llll(10t-4\rrrr).
$$
The partial derivative $u_t(x,t)$ at time $t=0$ has values
\begin{equation*}
u_t(x,0)=\llll.\dfrac{\partial^{1-\aaaa}F(x,t)}{\partial t^{1-\alpha}}\rrrr|_{t=0}=-4x^2(1-x).
\end{equation*}
Then 
$$u(x,\tttt)= u(x,0)+\tttt u_t(x,0)+O\llll(\tttt^2 \rrrr)$$
$$u(x,\tttt)= x^2(1-x)-4\tttt x^2(1-x)=x^2(1-x)(1-4\tttt)+O\llll(\tttt^2 \rrrr)$$
We obtain the second order approximation for the solution of equation \eqref{TFSE2} on the first layer on the grid $\mathcal{G}$.
\begin{equation}\label{GL1}
U_n^1=(nh)^2(1-nh)(1-4\tttt),\quad(n=1,2,\cdots,N-1)
\end{equation}
	\begin{table}[ht]
    \caption{Maximum error and order of numerical solutions \eqref{NS3} and \eqref{NS4} with approximation \eqref{GL1}   for the solution of  equation \eqref{TFSE2} on the first layer of $\mathcal{G}$,  at time $t=1$ when $\alpha=0.6$ and $\tttt=h$. }
    \begin{subtable}{0.5\linewidth}
      \centering
  \begin{tabular}{l c c }
  \hline \hline
    $h\;(\tttt=h)$ & $Error$ & $Order$  \\ 
		\hline \hline
$0.05\quad$         &$0.00051282$      &$1.35060$\\
$0.025\quad$        &$0.00019690$      &$1.38100 $\\
$0.0125\quad$       &$0.00007518$     &$1.38896$\\
$0.00625\quad$      &$0.00002862$     &$1.39337$\\
$0.003125\quad$     &$0.00001088$   &$1.39601$\\
		\hline
  \end{tabular}
    \end{subtable}%
    \begin{subtable}{.5\linewidth}
      \centering
				\quad\quad
  \begin{tabular}{ l  c  c }
    \hline \hline
    $\tttt\;(\tttt=h)$ & $Error$ &$Order$  \\ \hline \hline
$0.05\quad$     & $0.00001730$           & $2.28453$ \\
$0.025\quad$    & $3.85\times10^{-6}$          & $2.16657$ \\
$0.0125\quad$   & $9.02\times10^{-7}$          & $2.09523$ \\
$0.00625\quad$  & $2.17\times10^{-7}$          & $2.05667$ \\
$0.003125\quad$ & $5.29\times10^{-8}$          & $2.03510$ \\%  
\hline
  \end{tabular}
    \end{subtable} 
\end{table}
\section{Conclusion} In section 4 we compared the numerical solutions of the ordinary fractional relaxation equation and the partial fractional subdiffusion equation using approximations \eqref{A1} and \eqref{A2} for the Caputo derivative. The higher accuracy of approximation \eqref{A2} results in a noticeable improvement in the performance of the numerical solutions. Numerical experiments suggest that the numerical solutions converge to the exact solutions of the fractional relaxation and subdiffusion equations for all $\aaaa$ between $0$ and $1$. We are going to work on a proof for the convergence of the  numerical solutions discussed in section 4.


\begin{thebibliography}{99}
\bibitem{Cartea2007}  A. Cartea, D. del Castillo-Negrete,   Fractional diffusion models of option prices in markets with jumps.  Physica A, 374(2) (2007), 749--763.
\bibitem{Mainardi1996} F, Mainardi, Fractional relaxation-oscillation and fractional diffusion-wave phenomena, Chaos, Solitons $\&$ Fractals, 7(9) (1996), 1461 -- 1477.
\bibitem{MuslihAgrawalBaleanu2010} S. I. Muslih, Om P. Agrawal, D. Baleanu, A fractional Schrödinger equation and its solution, International Journal of Theoretical Physics, 49(8) (2010), 1746--1752.
\bibitem{Nigmatullin1986} R. R. Nigmatullin, The realization of the generalized transfer equation in a medium with fractal geometry,
 Physica Status Solidi B Basic Research, 133 (1986), 425--430.
\bibitem{LimaFordLumb2014}P. M. Lima, N. J. Ford, and P. M. Lumb, Computational methods for a mathematical model of propagation of nerve impulses in myelinated axons, Applied Numerical Mathematics 85, (2014), 38--53.
\bibitem{Lazarov2015_1} B.Jin, R. Lazarov and Z. Zhou, An Analysis of the $L1$ Scheme for the Subdiffusion Equation with Nonsmooth Data, arXiv:1501.00253, (2015).
\bibitem{Murio2008} D. A. Murio, Implicit finite difference approximation for time fractional diffusion equations , Computers \& Mathematics with Applications, 56(4) (2008), 1138 -- 1145.
\bibitem{ZhuangLiu2006} P. Zhuang, F. Liu, Implicit difference approximation for the time fractional diffusion equation, Journal of Applied Mathematics and Computing, 22(3) (2006), 87--99.
\bibitem{LinXu2007} Y. Lin, C. Xu, Finite difference/spectral approximations for the time-fractional diffusion equation, Journal of Computational Physics, 225 (2007), 1533--1552.
\bibitem{Sidi2004} A. Sidi, Euler-Maclaurin expansions for integrals with endpoint singularities: A new perspective. Numer. Math. 98 (2), (2004) pp. 371–387.
\bibitem{AbramowitzStegun1964}M. Abramowitz, I. A. Stegun, Handbook of Mathematical Functions with Formulas, Graphs, and Mathematical Tables. Dover, New York; 1964.
\bibitem{Diethelm2010} K. Diethelm,  The Analysis of Fractional Differential Equations: An Application-Oriented Exposition Using Differential Operators of Caputo Type. Springer; 2010.
\bibitem{Havil2003} J. Havil,  Gamma: Exploring Euler's Constant. Princeton, NJ: Princeton University Press; (2003).
\bibitem{Podlubny1999} I. Podlubny,  Fractional Differential Equations. Academic Press, San Diego; 1999.
\bibitem{MillerRoss1993} K.S. Miller, B. Ross,  An Introduction to the Fractional Calculus and Fractional Differential Equations. John Wiley \& Sons, New York; 1993.
\bibitem{Chen2012} G. Chen, Mean Value Theorems for Local Fractional Integrals on Fractal Space, Advances in Mechanical Engineering and its Applications 1, (2012), 5--8.
\bibitem{ChenDeng2014} M. Chen, W. Deng,  A second-order numerical method for two-dimensional two-sided space fractional convection diffusion equation, Applied Mathematical Modelling 38(13), (2014), 3244--3259.
\bibitem{DengLi2012}  W. Deng, C. Li, Numerical schemes for fractional ordinary differential equations, In: Miidla, P. (ed.) Numerical Modelling. InTech, Rijeka (2012), 355--374.
\bibitem{Dimitrov2013} Y. Dimitrov, Numerical Approximations for Fractional Differential Equations, Journal of Fractional Calculus and Applications, 5(3S), (2014), No. 22, 1--45. 
\bibitem{GulsuOzturkAnapali2013} M. G\"ulsu, Y. \"Ozt\"urk, and A. Anapal{\i}, Numerical approach for solving fractional relaxation-oscillation equation, Applied Mathematical Modelling 37 (8), (2013), 5927--5937.
\bibitem{GaoSunZhang2012} G. Gao, Z. Sun, and Y. Zhang, A finite difference scheme for fractional sub-diffusion equations on an unbounded domain using artificial boundary conditions, Journal of Computational Physics 231, (2012), 2865--2879.
\bibitem{GaoSunZhang2014} G. Gao, Z. Sun, and  H. Zhang, A new fractional numerical differentiation formula to approximate the Caputo fractional derivative and its applications, Journal of Computational Physics 259, (2014), 33--50.
\bibitem{Hasse1930} H. Hasse, Ein Summierungsverfahren f\"ur die Riemannsche Zeta-Reihe., Math. Z. 32, (1930), 458--464.
\bibitem{Lazarov2015_2} B.Jin, R. Lazarov and Z. Zhou, Two Fully Discrete Schemes for Fractional Diffusion and Diffusion-Wave Equations, 	arXiv:1404.3800, (2015). 
\bibitem{LiDengWu2011} C. Li, W. Deng, Y. Wu, Numerical analysis and physical simulations for the time fractional radial diffusion equation, Computers and Mathematics with Applications 62, (2011), 1024--1037.
\bibitem{KrishnaveniKannanBalachandar2013}K. Krishnaveni, K. Kannan, and S. R. Balachandar, Polynomial Approximation Method for Solving Composite Fractional Relaxation/Oscillation Equations, World Applied Sciences Journal 25 (12), (2013) 1789--1796.
\bibitem{LiDing2014} C. Li, H. Ding, Higher order finite difference method for the reaction and anomalous-diffusion equation, Applied Mathematical Modelling 38 (2014) 3802--3821.
\bibitem{LiChenYe2011}C. Li, A. Chen and J. Ye, Numerical approaches to fractional calculus and fractional ordinary
differential equation, Journal of Computational Physics 230, (2011) 3352--3368.
\bibitem{Sondow1994} J. Sondow,  Analytic Continuation of Riemann's Zeta Function and Values at Negative Integers via Euler's Transformation of Series, Proc. Amer. Math. Soc. 120, (1994), 421--424.
\bibitem{WeiChen2013}S. Wei, W. Chen, A Matlab toolbox for fractional relaxation-oscillation equations, arXiv:1302.3384, (2013).
\bibitem{Weierstrass1885} K. Weierstrass, \"Uber die analytische Darstellbarkeit sogenannter willk\"urlicher Functionen einer reellen
Ver\"anderlichen, Sitzungsberichte der Akademie zu Berlin, (1885), 633--639 and 789--805.
\bibitem{Weniger2007} E. J. Weniger  Asymptotic Approximations to Truncation Errors of Series Representations for Special Functions. in A. Iske and J. Levesley (Eds.), Algorithms for Approximation, (2007), 331--348.
\bibitem{YanPalFord2014}Y. Yan, K. Pal and N. J. Ford, Higher order numerical methods for solving fractional differential equations, BIT Numer Math 54, (2014),555--584.
\bibitem{ZengLiLiuTurner2013} F. Zeng, C. Li, F. Liu, and I. Turner,
 The Use of Finite Difference/Element Approaches for Solving the Time-Fractional Subdiffusion Equation, SIAM J. Sci. Comput., 35(6), (2013), A2976--A3000.
\end{thebibliography}
\end{document}